\documentclass[11pt]{article}

\usepackage{amsfonts, amsmath, amssymb}
\usepackage{graphicx}

\newtheorem{proposition}{Proposition}
\newtheorem{theorem}{Theorem}
\newtheorem{definition}{Definition}

\newenvironment{proof}{\smallskip\noindent\textbf{Proof.}\hspace{1pt}}%
{\hspace{-5pt}\qed\smallskip\goodbreak}

\newcommand{\qed}{\nobreak\quad\nobreak\hfill\nobreak$\square$\vspace{8pt}\par}

\renewcommand{\baselinestretch}{1.1}

\def\sqr#1#2{\vbox{\hrule height .#2pt
\hbox{\vrule width .#2pt height #1pt \kern #1pt
\vrule width .#2pt}\hrule height .#2pt }}
\def\square{\sqr74}

\def\bel{\begin{equation}\label}
\def\eeq{\end{equation}}
\newcommand{\R}{{\mathbb{R}}}
\newcommand{\N}{{\mathbb{N}}}

\def\L{{\mathbf L}}

\def\D{{\cal D}}
\def\forall{\hbox{for all}~}
\def\sgn{{\rm sign}}

\def\D{\mathcal{D}}

\def\O{\mathcal{O}}

\newcommand{\tv}{\mathrm{TV}\,}

\def\Hat{\widehat}

\begin{document}
\title{An Integro-Differential Conservation Law\\ arising in a Model of Granular Flow}

\author{Debora Amadori$^*$ and Wen Shen$^{**}$\\
  \\(*): Dipartimento di Matematica Pura \& Applicata,
  University of L'Aquila, Italy. \\E-mail: {\tt amadori@univaq.it} \\
  (**):  Department of Mathematics, Penn State University, U.S.A..  \\
  E-mail: {\tt shen\_w@math.psu.edu}}

\date{}

\maketitle

\begin{abstract}
  We study a scalar integro-differential conservation law. The
  equation was first derived in \cite{AS2} as the slow erosion limit
  of granular flow.  Considering a set of more general erosion
  functions, we study the initial boundary value problem for which one
  can not adapt the standard theory of conservation laws. We construct
  approximate solutions with a fractional step method, by recomputing
  the integral term at each time step. A-priori $\L^\infty$ bounds and
  BV estimates yield convergence and global existence of BV solutions.
  Furthermore, we present a well-posedness analysis, showing that the
  solutions are stable in $\mathbf{L}^1$ with respect to the initial
  data.
\end{abstract}

\section{Introduction}
We consider the initial boundary value problem for the scalar
integro-differential equation
\begin{equation}\label{eq1}
  q_t+\left( \exp \left\{ \int_x^0  f(q(t,\xi))\,d\xi \right\} \, f(q) \right)_x = 0\,, 
  \qquad  t\ge 0\,,\quad  x\le0\,,
\end{equation}
with initial condition
\begin{equation}\label{eq:init}
q(0,x)=  \bar q(x)\,,\qquad\qquad  x\le0\,.
\end{equation}

Note that the flux includes a non-local integral term. For notational
convenience, we introduce
\begin{eqnarray}
  K(q(t,\cdot))(x)&\dot=& \exp \left\{ \int_x^0  f(q(t,\xi))\,d\xi \right\} \,.\label{K}
\end{eqnarray}

The function $f:(-1,+\infty)\to \R \in C^2(\R)$ is called the {\em
  erosion function}.  The following assumptions apply to $f$:
\begin{equation} \label{eq:f} 
f(0)=0\,,\qquad f'>0\,,\qquad f''<0\,, \qquad 
\lim_{q\to-1}f(q) = -\infty \,,\qquad
  \lim_{q\to+\infty}\frac{f(q)}q = 0 \,.
\end{equation}
We remark that the characteristic speed of (\ref{eq1}) is 
$$ \dot x = f'(q) K \,.$$
By (\ref{K}) and (\ref{eq:f}), the characteristic speed is always
positive, therefore no boundary condition is assigned at $x=0$ for
(\ref{eq1}).

The equation (\ref{eq1}) arises as the {\em slow erosion limit} in a
model of granular flow, studied in \cite{AS2}, with a specific erosion
function
\begin{equation}\label{fex}
f(q) = \frac q{q+1}\,.
\end{equation}
Note that this function satisfies all the assumptions in (\ref{eq:f}).
In more details, let $h$ be the height of the moving layer, and $p$ be
the slope of the standing profile.  Assuming $p>0$, the following
$2\times2$ system of balance laws was proposed in \cite{HK}
\begin{equation}\label{eq2}
  \left\{\begin{array}{rl}
      \displaystyle h_t-(h p)_x&\displaystyle ~=~(p-1) h\,,\cr
      \displaystyle p_t+\big((p-1)h\big)_x&~=~0\,.\end{array}\right.
\end{equation}
This model describes the following phenomenon. 
The material is divided in two parts: a moving layer with height $h$ on top 
and a standing layer with slope $p >0$ at the bottom. 
The moving layer slides downhill with speed $p$.
If the slope $p=1$ (the critical slope), the moving layer passes through without
interaction with the standing layer. 
If the slope $p > 1$,  then grains initially at rest are hit by rolling 
grains of the moving layer and start moving as well. 
Hence the moving layer gets bigger. On the other hand, if $p < 1$, 
grains which are rolling can be deposited on the bed. 
Hence the moving layer becomes smaller.

In the slow erosion limit as $\|h\|_{\mathbf{L}^\infty}\rightarrow 0$,
we proved in \cite{AS2} that the solution for the slope $p$ in
\eqref{eq2} provides the weak solution of the following scalar
integro-differential equation
\begin{equation*}
  p_\mu+\left(\frac{p-1}{p}\cdot \exp\int_x^0 {\frac{p(\mu,y)-1}
      {p(\mu,y)}}\,dy \right)_x ~=~0\,.
\end{equation*}
Here, the new time variable $\mu$ accounts for the total mass of
granular material being poured downhill. Introducing $q\doteq p-1$ and
writing $t$ for $\mu$, we obtain the equation (\ref{eq1}) with
(\ref{fex}).

The result in \cite{AS2} provides the existence of entropy weak
solutions to the initial boundary value problem (\ref{eq1}) with $f$
given in (\ref{fex}) for finite ``time'' (which is actually finite
total mass).  However, well-posedness property was left open due to
the technical difficulties caused by the non-local term in the flux.
Furthermore, due to the discontinuities in $q$, the function
$k(t,x)=K(q(t,\cdot))(x)$ is only Lipschitz continuous in its
variables, therefore one can not apply directly previous results.
Indeed, classical results as \cite{Kru} require more smoothness on the
coefficients; see also \cite{CoMeRo}. Some closer results can be found
in \cite{KlRi,LTW} where the coefficient $k=k(x)$ does not depend on
time.

In this paper we consider a class of more general erosion functions $f$
that satisfy the assumptions in
(\ref{eq:f}), and we study existence and well-posedness of BV
solutions of (\ref{eq1}).  Assuming that the slope is always positive,
i.e., $q > -1$, we seek BV solutions with bounded total mass.
Therefore, we define $\D=\D_{C_0,\kappa_0}$ as the set of functions
that satisfy
\begin{eqnarray}\label{def:calD}
  \D_{C_0,\kappa_0} \doteq \left\{q(x):\ \  \inf_{x<0}\, q(x) \geq \kappa_0>-1\,, 
  \quad \tv\{q\}\leq C_0\,, \quad\| q\|_{\L^1(\R_-)}\leq C_0 \right\} \,.
\end{eqnarray}
Assume that the initial data satisfies $\bar q \in \D_{C_0,\kappa_0}$
for some constants $C_0>0,\kappa_0>-1$. A natural definition of
entropy weak solution is given below.

\begin{definition}\label{def1} Let $T>0$.
  A function $q$ is an \textbf{entropy weak solution} to \eqref{eq1}
  on $[0,T]\times\R_-$ with initial condition \eqref{eq:init}, if
the following holds.
\begin{itemize}
\item[\mbox{\bf (H1)}] $q:[0,T]\to\L^1(\R_-)\cap BV(\R_-)$, 
  $\inf_{x} q(t,x) >-1$,
  and the map $[0,T]\ni t\mapsto q(t)$ is Lipschitz in $\L^1(\R_-)$;
\item[\mbox{\bf (H2)}] $q$ is a weak entropy solution of the scalar
  conservation law
\begin{eqnarray}\label{eq:fixed_k-intro}
\left\{
  \begin{array}{l}
  q_t+\left( k(t,x)\, f(q)\right)_x =0\,,  \\[1mm]
  q(0,x)= \bar q(x)
\end{array}\right.
\end{eqnarray}
with $k$ defined by
\begin{eqnarray}
  k(t,x)&=&   K(q(t,\cdot))(x) ~=~ \exp \left\{ \int_x^0  f(q(t,\xi))\,d\xi \right\}\,.\label{k}
\end{eqnarray}
\end{itemize}
\end{definition}

Notice that, thanks to \textbf{(H1)}, the coefficient $k(t,x)$ in
\eqref{k} is Lipschitz continuous on $[0,T]\times\R_-$.

\smallskip
Now we state the main result of this paper.

\begin{theorem}\label{th:1} Assume \eqref{eq:f} and let $C_0>0$, $\kappa_0>-1$
  be given constants.  Then for any initial data $\bar q\in
  \D_{C_0,\kappa_0}$ there exists an entropy weak solution $q(t,x)$ to
  the initial-boundary value problem \eqref{eq1}--\eqref{eq:init} for
  all $t\ge 0$.  Moreover, consider two solutions $q_1(t,\cdot)$,
  $q_2(t,\cdot)$ of the integro-differential equation (\ref{eq1}),
  corresponding to the initial data
$$
q_1(0,x)= \bar q_1(x)\,,\qquad\qquad q_2(0,x) =\bar q_2(x)\,,\qquad\qquad
x<0\,,
$$
with $\bar q_1$, $\bar q_2 \in \D_{C_0,\kappa_0}$. Then for any $T>0$
there exists $L=L(T,C_0,\kappa_0)>0$ such that
\begin{eqnarray}\label{continous-dep-on-init-data}
  \|q_1(t,\cdot)- q_2(t,\cdot)\|_{\L^1(\R_-)}
  &\leq& {\rm e}^{Lt} \, \|\bar q_1-\bar q_2\|_{\L^1(\R_-)}\,,\qquad t\in [0,T]\,.
\end{eqnarray}
\end{theorem}

\medskip Recalling that $q=p-1=u_x-1$, the solution $q$ established by
Theorem~\ref{th:1} allows us to recover the profile $u$ of the
standing layer:
\begin{eqnarray}\label{u}
u(t,x) - x &=& \int_{-\infty}^x q(t,y)\,dy\,.
\end{eqnarray}
Moreover, since $K_x = - K f(q(t,x))$, the equation~\eqref{eq1} can be
rewritten as
\begin{eqnarray*}
  q_t - K_{xx}=0\,.
\end{eqnarray*}
Integrating in space on $(-\infty,x)$, using \eqref{u} and that
$K_x(q(t,\cdot))\in \L^1(\R_-)$, we arrive at
\begin{eqnarray*} 
  u_t - K_x &=& u_t + K f\left(u_x-1\right)
  ~=~0\,.
\end{eqnarray*}
This nonlocal Hamilton-Jacobi equation is studied in \cite{SZ}, with a
different class of erosion functions $f$.  Assuming more erosion for
large slope, i.e., $\lim_{q\to+\infty} f'(q) = \eta_0 >0$, the slope
$u_x$ of the standing layer would blowup, leading to jumps in the
standing profile $u$.  Notice that, in our case, only upward jumps in
$u_x$ can occur as singularities, which corresponds to convex
kinks in the profile $u$.

About the continuous dependence notice that, when $k$ is a prescribed
coefficient, the $\L^1$ stability estimate
(\ref{continous-dep-on-init-data}) holds with $L=0$, see
\eqref{L1-contraction}. On the other hand, for the integral equation
\eqref{eq1}, one cannot expect $L=0$ in general. Indeed, a small
variation in the $\L^1$ norm of the initial data may cause a variation
in the global term and then in the overall solution. However, a
special case in which (\ref{continous-dep-on-init-data}) holds with
$L=0$ is when $q_2\equiv 0$, which indeed is a solution of
\eqref{eq1}.

\smallskip Other problems involving a nonlocal term in the flux have
been considered in \cite{Daf88, Ch-Chr,CoHeMe}.  Well-known
integro-differential equations which lead to blow up of the gradients
include the Camassa-Holm equation \cite{CH} and the variational wave
equation \cite{BPZ}.
The Cauchy problem for (\ref{eq1}) with initial data with bounded support 
is studied in \cite{AS4} where we use piecewise constant approximation 
generated by front tracing and obtain similar results.

\smallskip The rest of the paper is structured in the following way.
As a step toward the final result, in Section 2 we study the existence
and well-posedness of the scalar equation (\ref{eq:fixed_k-intro}) for
a \textit{given} coefficient $k(t,x)$.  Here $k(t,x)$ is a local term,
and preserves the properties of the global integral term.  Such
equation does not fall directly within the classical framework of
\cite{Kru}, where more regularity on the coefficients is required
($C^1$).  In particular, BV estimates for solutions of
(\ref{eq:fixed_k-intro}) are needed to obtain the continuous
dependence on the initial data, see \eqref{dep_on_coefficients}. 
We employ a fractional step argument to deal with the time dependence
of $k$, and then follow an approach similar to \cite{BJ} (see also
\cite{Gue04}), where the authors deal with the case of
$k=k(x)\in\L^\infty$.  We further refer to \cite{CoMeRo} on total
variation estimates for general scalar balance laws: their result, in
our context, would require more regularity ($C^1$) on the coefficient
$k$.

The properties of the integral operator $K$, defined at \eqref{K}, are
summarized in the last Appendix.

\section{Local well-posedness of solutions with a given coefficient $k$}
\label{sec:given-coeff}
\setcounter{equation}{0}

In this section we study the well-posedness of the scalar equation
(\ref{eq:fixed_k-intro}) for a \textit{given} coefficient $k(t,x)$, by
reviewing some related results and completing the arguments where
needed.

Throughout this section, we will use $u$ as the unknown
variable. Consider
\begin{eqnarray}\label{eq:fixed_k}
&& u_t + \Big(k(t,x)f(u)\Big)_x=0\,, \qquad x\le 0, \quad t\ge 0 \\
\label{eq:initU}
&&  u(0,x)=\bar u(x)\,, \qquad\qquad\qquad    x<0
\end{eqnarray}
where $k(t,x)$ satisfies the following assumptions, for some $T>0$:

\medskip
\begin{tabular}{ll}
  &\quad $k(t,x)\in\L^\infty\left([0,T]\times \R_-\right)$, 
  it is Lipschitz continuous and $\inf_{t,x} k >0$;\\[2mm]
  {\rm \textbf{(K)}}&\quad $\tv \{k(t,\cdot)\}$, $\tv \{k_x(t,\cdot)\}$ 
  are bounded uniformly in time;\\[2mm]
  &\quad $[0,T]\ni t\to k_x(t,\cdot)\in{\L^1(\R_-)} $ is Lipschitz continuous.
\end{tabular}

\smallskip\par\noindent The above assumptions on $k$ are motivated by
the properties of the integral operator $K$, see
Proposition~\ref{properties_of_k} in the Appendix.

\begin{theorem}\label{properties_semigroup_k_fixed}
  Assume $f$ satisfies \eqref{eq:f} and $k(t,x)$ satisfies {\rm
    \textbf{(K)}}. Let $C_0>0$, $\kappa_0>-1$ be given constants.
  Then there exist two constants $C_1$ and $\kappa_1$, with possibly
  $C_1 \ge C_0$ and $-1 < \kappa_1\le \kappa_0$, and an operator
  $P:[0,T]\times\D_{C_0,\kappa_0}\to\D_{C_1,\kappa_1} $ such that:
\begin{itemize}
\item[1)] the function $u(t,x) = P_{t}(\bar u)$ is a weak entropy
  solution of \eqref{eq:fixed_k} with initial data $u(0,\cdot) = \bar
  u \in \D_{C_0,\kappa_0}$\,;
  
\item[2)] for any $\bar u_1$, $\bar u_2\in \D_{C_0,\kappa_0}$ one has
\begin{eqnarray}\label{L1-contraction}
  \| P_{t}(\bar u_1) - P_{t}(\bar u_2) \|_{\L^1(\R_-)} 
  &\le& \|\bar u_1 - \bar u_2 \|_{\L^1(\R_-)} \,.
\end{eqnarray}
\end{itemize}
\end{theorem}

\begin{proof}
  Let $\bar u\in \D_{C_0,\kappa_0}$. We introduce the parameter
  $\Delta t>0$ and define $t_n=n\Delta t$ for any integer $n\ge 0$.
  We approximate the coefficient $k$ by
\begin{eqnarray}\label{k_Deltat}
  k_{\Delta t}(t,x)&=& k(t_n, x)\qquad  (t,x)\in [t_n,t_{n+1})\times  \R_-\,,
  \quad n\ge 0
\end{eqnarray}
which is constant in time on each interval $[t_n,t_{n+1})$,
and consider the equation
\begin{equation}\label{eq:fixed_k-Dltat}
  u_t + \Big(k_{\Delta t}(t,x)f(u)\Big)_x=0\,.
\end{equation}
By adapting the analysis in \cite{BJ}, on each interval
$[t_n,t_{n+1})$ the entropy solution for (\ref{eq:fixed_k-Dltat}),
call it $u_{\Delta t}$, exists and the corresponding operator $(t,\bar
u)\mapsto u_{\Delta t}(t,\cdot)$ is contractive in $\L^1(\R_-)$,
provided that $u_{\Delta t}$ is bounded from both below and above.
Furthermore, the \textit{complete flux}
\begin{eqnarray*}
F(t,x) &\dot= & k_{\Delta t}(t,x) f(u_{\Delta t}(t,x))
\end{eqnarray*}
has the following properties: its sup norm does not increase in time,
\begin{eqnarray}\label{decay_sup_flux}
|F(t,x)|&\le& \sup |F(t_n,\cdot)|\,,\qquad 
  t\in (t_n, t_{n+1})\,,
\end{eqnarray}
as well as its total variation:
\begin{eqnarray}\label{decay_tv_flux}
  \tv \{F(t,\cdot)\} &\le&  \tv \{F(t_n,\cdot)\}\,,\qquad 
  t\in (t_n, t_{n+1})\,.
\end{eqnarray}

\smallskip
We now establish the lower and upper bounds for $u_{\Delta t}$.
For notation simplicity, in the following we denote by $k(t,x)$ and
$u(t,x)$ the approximate coefficient and solution respectively,
without causing confusion.
We define the constants $k_0$, $L$, $L_1$ such
that, recalling {\rm \textbf{(K)}}, one has:
\begin{eqnarray}
&& k_0= \inf_{t,x} k >0\,;\label{KK1}\\
&&|k(t_1,x_1) -k(t_2,x_2)|\le L \left(|t_1-t_2| + |x_1-x_2|
  \right)\quad   \forall t_i,\ x_i,\ i=1,\ 2\,;\label{KK2}\\[2mm]
&& \tv \{k(t_1,\cdot) - k(t_2,\cdot)\} = \|k_x(t_1,\cdot) -
  k_x(t_2,\cdot) \|_{\L^1(\R_-)} \le L_1 |t_1 - t_2|\label{KK3}
\end{eqnarray}
and set $L_2=L/k_0$.
We first give some formal arguments. The evolution of the complete flux $F=kf(u)$ along the
characteristic $x(t)$ with $\dot x = f'(u) k$ follows the equation
\begin{equation}\label{Fchar}
\frac{d}{dt} F (t,x(t)) = (kf)_t + f'k (kf)_x = k_t f = \frac{k_t}{k} F\,.
\end{equation}
By our assumptions (\textbf{K}), the term $k_t/k$ is uniformly bounded.
Therefore, $|F|$ grows at most at an exponential rate, and remains bounded
for finite time $t\le T$. 
Therefore  $|f(u)|$ remains bounded as well.
By the $4^{th}$ assumption in (\ref{eq:f}), 
$u$ never reaches $-1$ in finite time, leading to a lower bound on $u$. 

The same argument leads to an upper bound for $f(u)$, if $f(u) \to
+\infty$ as $u\to +\infty$. 
However, if $f(u) \to f_0 >0$ as $u\to +\infty$, we need a different argument.  
We observe that, along a characteristic $x(t)$, one has
\begin{eqnarray}\label{eq:u_across_chars}
\frac{d}{dt} u(t,x(t)) = -k_x(t,x) f(u) \,. 
\end{eqnarray}
By the lower bound on $u$, the growth of $u$ remains uniformly bounded, 
yielding an upper bound.

We now make these arguments rigorous for the approximate
solutions. At time $t=0$ one has
\begin{eqnarray}\label{eq:bounds_initial-time}
|k(0,x) f(\bar u(x))| ~ \le ~ C_1
\end{eqnarray}
for some $C_1 \ge 0$ that depends on the bounds for $k$ and $\bar u$.  
We claim that, as long as the approximate solution exists, we have
\begin{eqnarray}\label{eq:bound-total-flux}
| k(t,x) f(u(t,x)) | ~ \le ~ C_1 {\rm e}^{L_2 t}
\,.
\end{eqnarray}
Indeed, by \eqref{eq:bounds_initial-time} and \eqref{decay_sup_flux},
the inequality \eqref{eq:bound-total-flux} is valid on $[0,t_1)$.
Assume now that \eqref{eq:bound-total-flux} is valid on $[0,t_{n+1})$,
$n\ge 0$, i.e.,
\begin{equation}\label{fk}
\left| F(t,x)\right| = \left| k(t_{n},x) f(u(t_{n},x))\right| \le C_1\, {\rm e}^{L_2 t_n}\,,
\qquad t\in[t_n,t_{n+1})\,.
\end{equation}
At time $t=t_{n+1}$ one has
\begin{eqnarray*}
 && \hspace{-1cm} 
 \left| k(t_{n+1},x) f(u(t_{n+1},x))\right| ~=~ 
  \frac{k(t_{n+1},x)}{k(t_{n},x)}\, \left| k(t_{n},x) f(u(t_{n+1},x))\right| \\
  &\le & 
  \left(1 + \frac{L}{k_0}\Delta t \right) 
\cdot \sup_x \left| k(t_{n},x) f(u(t_{n+1},x))\right|
  ~\le ~ {\rm e}^{L_2 \Delta t} \cdot  C_1 \, {\rm e}^{L_2  t_{n}}
  ~= ~ C_1 \, {\rm e}^{L_2  t_{n+1}}
\end{eqnarray*}
By induction, this proves (\ref{eq:bound-total-flux}), which in turn
gives the lower bound $\kappa_1$ for $u$. The upper bound also follows
if $f(u) \to +\infty$ as $u\to +\infty$.

Finally, we consider the case that $f(u) \to f_0 >0$ as $u\to
+\infty$.  At any given point $(\bar t, \bar x)$ one can trace back
along an extremal backward generalized characteristic, which is
classical on each $(t_n,t_{n+1})$ and continuous up to $t=0$.  Since
now the r.h.s. of \eqref{eq:u_across_chars} is bounded, then $u$ grows
at a linear rate, and therefore remains bounded.

We remark that the lower bound yields an a-priori bound on the wave
speed. Indeed, since $f'$ is a decreasing function, the characteristic
speed is bounded,
\begin{eqnarray*}
  \lambda = k f'(u) &\le&  \|k\|_\infty\, f'\left(\kappa_1\right)\,.
\end{eqnarray*}

\smallskip
\noindent\textbf{Bound on total variation.} 
We estimate the total variation of $F(t,x) = k(t,x) f(u(t,x))$.  On
the interval $(t_n,t_{n+1})$ the coefficient $k$ is constant in time
and we use \eqref{decay_tv_flux}.  On the other hand, the total
variation might increase at $t_n$ when $k$ is updated.  
Then we observe that
\begin{eqnarray}\label{eq:updated-flux-frac-step}
  F(t_n,x)&=& 
 \left[1+\frac{k(t_n,x)- k(t_{n-1},x)}{k(t_{n-1},x)}\right]\,F(t_n-,x)\,,
\end{eqnarray}
therefore
\begin{eqnarray}
  \tv \{F(t_n,\cdot)\} &\le& 
  \left(1+\frac{\|k(t_n,\cdot)- k(t_{n-1},\cdot)\|_\infty}{\inf k(t_{n-1},\cdot)}
  \right)\,\tv \{F(t_{n}-,\cdot)\} \nonumber\\
  &&\qquad +~\sup |F|\cdot \tv \left\{\frac{k(t_{n},\cdot) -k(t_{n-1},\cdot)}{k(t_{n-1},\cdot)}\right\}\,.
  \label{tilde-f-fract-step}
\end{eqnarray}
Thanks to \eqref{KK1}--\eqref{KK3}, we have 
$$
  \frac{\|k(t_n,\cdot)- k(t_{n-1},\cdot)\|_\infty}{\inf k(t_{n-1},\cdot)}~\le ~ L_2 \Delta t\,, 
  \qquad
    \tv \left\{\frac{k(t_{n},\cdot)-k(t_{n-1},\cdot)}{k(t_{n-1},\cdot)}\right\}
  ~\le~ L_3 \,\Delta t\,,
$$
for a suitable constant $L_3$ independent on $\Delta t$. Moreover
$F=kf$ is uniformly bounded thanks to {\rm \textbf{(K)}} and the
bounds on $u$. Hence we conclude that
\begin{eqnarray*}
  \tv \{F(t_n,\cdot)\} &\le& 
\left(1+L_2\Delta t \right)\,\tv \{F(t_{n-1},\cdot)\}  ~+~ L_4\Delta t
\end{eqnarray*}
for a suitable $L_4>0$. By induction it follows that 
\begin{eqnarray*}
  \tv \{F(t,\cdot)\} &\le& {\rm e}^{L_2 t}\tv \{F(0+,\cdot)\} ~+~ \frac{L_4}{L_2}\left({\rm e}^{L_2 t}-1 \right)\,. 
\end{eqnarray*}
Recalling that $f(u)= F / k$, one obtains the 
$BV$ bound for $f(u(t))$,
\begin{eqnarray*}
  (\inf f')\, \tv \{u(t,\cdot)\} ~\le~ \tv \{f(u(t,\cdot))\} &\le& \frac{1}{\inf k} \tv\{ F(t,\cdot)\} ~+~ 
  \frac{\|F \|_\infty}{(\inf k)^2} \tv \{k(t,\cdot)\}\,.
\end{eqnarray*}
This gives a bound on the total variation for  $u(t)$:
\begin{eqnarray}\label{bound-on-tv-u}
  \tv \{u(t)\} &\le& C\left[\tv \{F(t,\cdot)\}+\tv \{k(t,\cdot)\}\right]~\le~C_1(t)
\end{eqnarray}
where the constant $C$ depends on $\inf_x u$, $\sup_x u$, $\inf_x k$,
$\sup_x k$. Hence the total variation of $u$ may increase in time but
it remains bounded as long as $u$ remains bounded.

\medskip\par
Taking the limit $\Delta t\to 0$, the coefficient $k_{\Delta t}$ converges uniformly
to $k$. Correspondingly, the family $u_{\Delta t}$ converges (up to a
subsequence) to a weak solution $u$ of the original equation,
satisfying the same upper and lower bounds and \eqref{bound-on-tv-u}.

Moreover, in the limit as $\Delta t\to 0$, the Kru\v{z}kov entropy
inequalities for equation~\eqref{eq:fixed_k}
\begin{eqnarray}\label{kruzkov}
\partial_t |u-\alpha| ~+~ \partial_x \left[k(x,t) |f(u)-f(\alpha)|\right] 
~+~ \sgn(u-\alpha) k_x(x,t) f(\alpha) &\le& 0
\end{eqnarray}
for all $\alpha \in\R$, hold in the sense of distributions.
\end{proof}

Next we establish the continuous dependence on the coefficient
function.  We rely on a result in \cite{KlRi} (Corollary 3.2) that
applies to Cauchy problems and to the case of $k=k(x)$, that is, the
coefficient does not depend on time.

For convenience of the reader we report that statement of \cite{KlRi}
adapted to our situation. Consider the two equations
\begin{eqnarray}
  u_t + \Big(kf(u)\Big)_x&=&0\,, \qquad \quad t\ge 0 \,, \label{eq:fk1}\\
  u_t + \Big(\tilde kf(u)\Big)_x&=&0\,, 
  \qquad \quad t\ge 0 \,. \label{eq:fk2}
\end{eqnarray}

\begin{proposition}\label{Prop:KlaRis}
  For $x\in\R$, let $k(x)$, $\tilde k(x)\in BV(\R)$ satisfy
\begin{eqnarray*}
  k_x\,,\  \tilde k_x \in BV(\R)\,;\qquad \inf k\,,\ \inf \tilde k \ge \alpha >0 
\end{eqnarray*}
for some positive $\alpha$. Consider the initial data $u_0$, $\tilde
u_0\in BV(\R)$ for the two equations \eqref{eq:fk1}, \eqref{eq:fk2}
respectively and let $u(t,x)$, $\tilde u(t,x)$ be the corresponding
solutions, assuming that they are bounded from above and bounded away
from $-1$. Let $C_1$ be a bound on $|f|$ over the range of the
solutions. Then
\begin{eqnarray}\nonumber
  &&  \|u(t,\cdot) - \tilde u(t,\cdot)\|_{\L^1(\R)}~\le~  
  \|u_0 - \tilde u_0\|_{\L^1(\R)}  \\
  && \qquad\qquad +~  t \left\{C_1 \tv\{k-\tilde k\} +  C_2 \left(1+\tv u_0 + 
      \tv\tilde u_0 \right)\|k-\tilde k\|_\infty\right\}\label{estim:KlaRis}
\end{eqnarray}
where $C_2$ depends on the bounds on $u$, $k$, $\tv\{k\}$ and on $\tilde u$, $\tilde k$, 
$ \tv \{\tilde k\}$.
\end{proposition}

The continuous dependence property 
for our problem 
follows from Proposition~\ref{Prop:KlaRis}, by properly extending the
IBVP into Cauchy problems.

\begin{theorem}\label{3}
  For $x<0$, let $k(t,x)$, $\tilde k(t,x)$ satisfy the assumption
  \textbf{(K)}, and assume that the initial data $\bar u$ belongs to
  $\D_{C_0,\kappa_0}$ (defined at \eqref{def:calD}). Let $u(t,\cdot)$,
  $\tilde u(t,\cdot)$ be the solutions of the conservation laws
  \eqref{eq:fk1}, \eqref{eq:fk2} respectively, with the same initial
  data $\bar u$, for some time interval $[0,T]$ ($T>0$).

Then, the following estimate holds
\begin{eqnarray}\nonumber
&& \frac 1t \| u(t,\cdot) - \tilde u(t,\cdot) \|_{\L^1(\R_-)} ~\le~
    \Hat C_1 \sup_{t\in[0,T]}\tv\left\{k(t,\cdot)- \tilde k(t,\cdot) \right\} \\
&&\qquad +~ \Hat C_2 \left(1+\sup_\tau \tv u(\tau,\cdot) + 
      \sup_\tau \tv\tilde u(\tau,\cdot) \right)\|k- \tilde k\|_{\L^\infty([0,t]\times\R_-)}  
\label{dep_on_coefficients}
\,,
\end{eqnarray}
where $\Hat C_1$ is a bound on $|f|$ over the range of the solutions
and $\Hat C_2$ depends on the bounds on the solutions, the
coefficients and their total variation $\tv\{k(t,\cdot\}$,
$\tv\{\tilde k(t,\cdot\}$.
\end{theorem}

\begin{proof}
The IBVP \eqref{eq:fixed_k}--\eqref{eq:initU} can be extended to 
  the following  Cauchy problem
\begin{equation}\label{eq:fkc}
  u_t + \Big(k(t,x)f(u)\Big)_x=0\,, 
  \qquad x\in \R, \quad t\ge 0 \,,
\end{equation}
with extended initial data 
\begin{eqnarray}\label{initc}
  u(0,x)&=&\left\{ \begin{array}{ll} \bar u(x) \quad  &\mbox{for}~ x\le0\,, \\
      \bar u(0-)       & \mbox{for}~x>0 \end{array}\right. 
\end{eqnarray}
and the extended coefficient function $k(t,x)$
$$k(t,x) ~=~  \lim_{y\to 0-}k(t,y)\qquad \mbox{for}~x>0\,.$$

Due to the fact that the characteristic speed is positive, the
solution for the Cauchy problem \eqref{eq:fkc}--\eqref{initc}
restricted on $x\le0$ will match the solution for the IBVP
\eqref{eq:fixed_k}.

In a same way, the IBVP \eqref{eq:fk2} is extended to the Cauchy problem for
\begin{equation}\label{eq:fkc2}
  u_t + \Big(\tilde k(t,x)f(u)\Big)_x=0\,, 
  \qquad x\in\R\,, \quad t\ge 0 
\end{equation}
with data \eqref{initc}. Without causing confusion, let's still denote
$u(t,x)$ and $\tilde u(t,x)$ the solutions for \eqref{eq:fkc} and
\eqref{eq:fkc2}, respectively, and let $u_{\Delta}(t,x)$ and $\tilde
u_{\Delta}(t,x)$ be the corresponding approximate solutions,
constructed in the same way as in the proof of
Theorem~\ref{properties_semigroup_k_fixed}, with approximate
coefficients $k_{\Delta t}$ and $\tilde k_{\Delta t}$ as in
\eqref{k_Deltat}.

Denote the distance between these two solutions by
$$ 
e_\Delta(t) ~\doteq ~\left\| u_\Delta (t,\cdot)- \tilde u_\Delta
  (t,\cdot) \right\|_{\L^1(\R)}\,.
$$
Notice that $e_\Delta(0)=0$ and that $e_\Delta(t) \ge \left\| u_\Delta
  (t,\cdot)- \tilde u_\Delta (t,\cdot)\right\|_{\L^1(\R_-)}$\,.

\smallskip On each time interval $[t_n,t_{n+1})$ the coefficient is
constant in time and the assumptions of Proposition~\ref{Prop:KlaRis}
are satisfied.  Hence, from \eqref{estim:KlaRis}, we have the
following estimate
\begin{eqnarray}\nonumber
  &&  \hspace{-1cm} e_\Delta(t_{n+1}) - e_\Delta(t_n)~\le~ \Delta t~
  \Hat C_1  \tv_{\R}\left\{k_{\Delta t}(t_n,\cdot)- \tilde k_{\Delta t}(t_n,\cdot)\right\}\\
  && +~ \Delta t~ \Hat C_2 \left(1 + \tv_{\R_-} u(t_n,\cdot) +  
    \tv_{\R_-}\tilde u(t_n,\cdot) \right)
  \left\| k_{\Delta t}(t_n,\cdot)- \tilde k_{\Delta t} 
    (t_n,\cdot)\right\|_{\L^\infty(\R)}
\label{err1}
\end{eqnarray}
for some constants $\Hat C_1$ and $\Hat C_2$ that are uniform on
$[0,T]$.  Notice that, in the above lines, $\tv_{\R}\left\{k_{\Delta
    t}- \tilde k_{\Delta t}\right\}$ coincides with $\tv_{\R_-}$ of
the same quantity and, similarly, the $\L^\infty$-norm on $\R$
coincides with the $\L^\infty$-norm on $\R_-$. Concerning $\tv_{\R} u$
(similarly for $\tv_{\R} \tilde u$), we replaced it with $\tv_{\R_-}
u$ with an error that is bounded and possibly depending on $T$.

\smallskip
Summing up \eqref{err1} in $n$, we get
\begin{eqnarray*} 
&&  \hspace{-1cm} 
e_\Delta(t_N)-e_\Delta(0) ~=~  \sum_{n=0}^{N-1} e_\Delta(t_{n+1}) - e_\Delta(t_n)
~\le~  t_N \Hat C_1 \sup_{t\in[0,t_N]}\tv_{\R_-}\left\{k_{\Delta t}- \tilde k_{\Delta t} \right\}  
\\
&&\qquad\qquad   +~ t_N \Hat C_2~ 
\left(1 + \sup_t \tv_{\R_-} u(t,\cdot) +  \sup_t\tv_{\R_-}\tilde u(t,\cdot) \right)
\left\| k_{\Delta t}- \tilde k_{\Delta t} \right\|_{\L^\infty([0,t_N]\times\R_-)}\,.\\
\end{eqnarray*}
Now taking the limit $\Delta t \rightarrow 0$, we get
\eqref{dep_on_coefficients}, completing the proof.
\end{proof}

\section{Well-posedness of the integro-differential equation}
\label{Sec:3}
\setcounter{equation}{0}

In this section we prove the main Theorem~\ref{th:1}. 
In Subsection~\ref{subsec:integral-eq} we define a family of approximate
solutions to \eqref{eq1}--\eqref{eq:init} and show their compactness,
locally in time. Then in Subsection~\ref{subsec3.1} 
we show that the limit solution can be prolonged
beyond the existence time, by improving the estimates on upper and
lower bound for the exact solution of \eqref{eq1}--\eqref{eq:init}.
Finally, in Subsection~\ref{Subsec:3} we show that the flow generated
by the integro-differential equation (\ref{eq1}) is Lipschitz
continuous, restricted to any domain $\D$ given at \eqref{def:calD}.

\subsection{Local in time existence of BV solutions}
\label{subsec:integral-eq}

In this Subsection we prove the following existence theorem.
\begin{theorem}\label{th:4}
  Let $C_0$, $\kappa_0$ be given constants and let $\bar q(x)\in
  \L^1(\R_-)\cap BV(\R_-)$ such that
\begin{itemize}
\item[(a)] $\inf_{x<0}\, \bar q(x) ~\geq~ \kappa_0>-1$\,;
\item[(b)] $\tv\{\bar q(\cdot)\}\leq C_0$\,;
\item[(c)] $\|\bar q\|_{\L^1(\R_-)}\leq C_0$\,.
\end{itemize}
Then there exist $T>0$, $\kappa_1>-1$ and $C_1>0$ such that
\begin{eqnarray}\label{the-problem}
\left\{
\begin{array}{l}
q_t+\left( 
\exp \left\{ \int_x^0  f(q(t,\xi))\,d\xi   \right\}
\, f(q)  \right)_x =0\,,  \\[2mm]
  q(0,x)=  \bar q(x)\,,
\end{array}
\right.
\end{eqnarray}
admits an entropy weak solution $q(t,x)$ on $[0,T]\times\R_-$ that
satisfies
\begin{itemize}
\item[(a)'] $~\inf_{x<0}\, q(t,x) ~\geq~ \kappa_1>-1$\,;
\item[(b)'] $~\tv\{q(t,\cdot)\}\leq C_1$\,;
\item[(c)'] $~\|q(t,\cdot)\|_{\L^1(\R_-)}\leq\|\bar  q\|_{\L^1(\R_-)}$\,.
\end{itemize}
\end{theorem}

\begin{proof}
  We define a sequence of approximate solution to the scalar
  equation~\eqref{eq1}--\eqref{K}.  We fix $\Delta t>0$ and set
  $t_n=n\Delta t$, $n\in\N$.  The approximation is generated
  recursively, as $n$ starts from 0 and increases by 1 after each
  step.
%
  For each step with $n\ge 0$, let $q(t,x)$ be defined on
  $[0,t_n)\times \R_-$ and set
\begin{eqnarray*}
  k_n(x) &\dot =& \exp \left\{ \int_x^0
    f(q(t_n,\xi))\,d\xi \right\} \,.
\end{eqnarray*}
Then we define $q$ on $[t_n, t_{n+1})\times \R_-$ as the solution of
the problem
\begin{eqnarray*}
\left\{
\begin{array}{ll}
  q_t+\left( k_n(x)\, f(q) \right)_x
  =0\,,&\qquad t\in[t_n, t_{n+1})\\[1mm]
  q(t_n,x)=  q(t_n -,x)
  \,.&
\end{array}\right.
\end{eqnarray*}
This procedure leads to a solution operator $t\mapsto S^{\Delta
  t}_t\bar q = q^{\Delta t}(t,\cdot)$, defined up to a certain time
$T=T(\Delta t,\bar q)>0$, of the problem
\begin{eqnarray}\label{eq1-Deltat}
\left\{
\begin{array}{ll}
q_t+\left( k^{\Delta t}(t,x)\, f(q)
\right)_x
=0\,,  &\qquad t>0\\[1mm]
q(0,x)= \bar q(x)\,,&
\end{array}\right.
\end{eqnarray}
where $k=k^{\Delta t}$ is defined by
\begin{eqnarray}\label{k-Deltat}
  k^{\Delta t}(t,x) &=& \sum_{n\ge 0} \, \chi_{[t_n, t_{n+1})}(t) \cdot k_n(x)\,. 
\end{eqnarray}
Notice that the operator $S^{\Delta t}_t$ has the semigroup property
$S^{\Delta t}_{\tau_1+\tau_2} =S^{\Delta t}_{\tau_1} S^{\Delta
  t}_{\tau_2}$ for $\tau_1$, $\tau_2\in (\Delta t) \N$.  Now we prove
uniform bounds, independent of $\Delta t$, on the family of
approximate solutions.

\paragraph{The $L^1$ bound.} This follows by the application of
\eqref{L1-contraction} in Theorem~\ref{properties_semigroup_k_fixed},
at each time step $[t_n,t_{n+1})$, and the fact that $t\mapsto
q(t,\cdot)$ is continuous in $\L^1$. Until the solution is defined, we
have
\begin{eqnarray}\label{}
\| q(t,\cdot)\|_{\L^1} &\le& \| q(0,\cdot)\|_{\L^1} \,.
\end{eqnarray}

\paragraph{Lower and upper bound on $q$.} 
Define 
$$z(t)~=~ \inf_x\, q(t,x)\,, \qquad 
w(t)~=~ \sup_x\, q(t,x)\,.
$$
We observe that, by comparison with the equilibrium solution $u\equiv 0$, 
(i) if $z(0) \ge 0$ then $z(t)\ge 0$;
and (ii) if $w(0) \le 0$ then $w(t) \le 0$ for all $t>0$. 

Now consider $-1 < z(0) < 0$ and $w(t) >0$.  Choose $\delta$ and $M$
such that $z(0)\ge -1 + 2\delta$ and $w(0) \le M/2$.  For example, one
can take $\delta=(\kappa_0+1)/2$ and $M= 2 w(0)$.  Let
$T=T(\delta,M)>0$ be the first time that one of the following bounds
fails,
\begin{equation}\label{claim}
z(t) \ge  -1 + \delta\,, \qquad w(t) \le M\,.
\end{equation}

Then, for $t\le T$, {from} the analysis of equation \eqref{eq1-Deltat}
(see (\ref{eq:u_across_chars})), we find that $z$ and $w$ are
continuous and satisfy
\begin{eqnarray}
  z(t) &\ge& z(0) + \sup_{x}{\left|k^{\Delta t}_x(t,x)\right|}
  \int_{0}^t f(z(\tau))\,d\tau\,, \qquad z<0\,, 
 \label{eq-for-z}\\
 w(t) & \le & w(0) + \sup_{x}{\left|k^{\Delta t}_x(t,x)\right| }
  \int_{0}^t f(w(\tau))\,d\tau\,, \qquad w>0\,. 
  \label{eqW}
\end{eqnarray}
Note that in (\ref{eq-for-z}) we have $f(z)\le0$, and in (\ref{eqW}) we have $f(w)\ge 0$. 
For $\left|k^{\Delta t}_x\right|$, we have the estimate
\begin{eqnarray}
  \left|k^{\Delta t}_x(t,x) \right|
  &=&  \left|k^{\Delta t}(x)f(q(t_n, x))\right|
  ~\le~  \exp\left\{\int_x^0 \left| f(q(t_n,\xi))\right| \, d\xi \right\}
  f(M) \nonumber \\
  &\le & f(M) \exp\{f'(-1+\delta) \left\| \bar q\right\|_{\L^1} \} ~\le~ C(\delta,M)\,.
  \nonumber
\end{eqnarray}
This gives us
\begin{eqnarray}
  z(t) &\ge& z(0) + C(\delta,M) \int_{0}^t f(z(\tau))\,d\tau 
        ~ \ge~ z(0) - C(\delta,M) t \left|f'(-1+\delta)\right| \,, 
\nonumber
\\
  w(t) &\le & w(0) + C(\delta,M) \int_{0}^t f(w(\tau))\,d\tau
        ~ \le ~ w(0) + C(\delta,M) t f(M)\,.
\nonumber
\end{eqnarray}
We conclude that the bounds in (\ref{claim}) hold for $t\le T$ with 
\begin{equation}\nonumber
T(\delta,M) = \min\{T_1,T_2\}\,,
\end{equation}
where
\begin{equation}\nonumber
T_1(\delta,M) = \frac{\delta}{C(\delta,M) \left|f'(-1+\delta)\right| }\,,
\qquad
T_2 (\delta,M) = \frac{M/2}{C(\delta,M) f(M)}\,,
\end{equation}
yielding the lower and upper bounds. 

Finally, if $z(0)\ge0$ and $ w(0)>0$, or if $z(0)<0$ and $w(0)\le0$,
then we would only need to establish one of the bounds in
(\ref{claim}), and the result follows.

\paragraph{Bounds on $f,f',k$.} 
Once we have a lower, upper bound on $q$ and the bound on
$\|q\|_{\L^1}$, we immediately find that
\begin{eqnarray}\label{bounds-on-f-fprime}
f(q(t,x))\,,\quad f'(q(t,x))\,,\quad \int_x^0 f(q(t,\xi))\,d\xi~~~ \in~~~
\L^\infty\left([0,T]\times \R_- \right)
\end{eqnarray}
uniformly w.r.t. $\Delta t$. By definition \eqref{k-Deltat} of $k$, we
can easily verify that the following properties hold uniformly
w.r.t. $\Delta t$:

\begin{itemize}
\item[(i)] $k\in\L^\infty\left([0,T]\times \R_-\right)$,  $\inf_{t,x} k >0$;
\item[(ii)] $k_x \in \L^\infty\left([0,T]\times \R_-\right)$\,;
\item[(iii)] $\tv k(t,\cdot)$ is bounded uniformly in time\,.
\end{itemize}
Indeed, (i) follows from the definition of $k$ and
\eqref{bounds-on-f-fprime}. About (ii), at each time $t$ we have
$k(t,\cdot) = k_n(\cdot)$ for some $n$, and $k_x= -k_n
f(q(t_n,\cdot))$. Then $k_x \in \L^\infty$ because of (i) and
\eqref{bounds-on-f-fprime}. Finally
\begin{eqnarray*}
  \tv k(t,\cdot)&=&\| k_x \|_{\L^1} ~=~ \| k_n f(q(t_n,\cdot))
  \|_{\L^1}
  ~\le~ M\,\|k\|_\infty\, \|q(t_n,\cdot)\|_{\L^1}
  ~\le~ M\,\|k\|_\infty \|\bar q\|_{\L^1}
\end{eqnarray*}
where $M= \sup f'$, that depends on the lower bound on $q$.

\smallskip Lastly, from (i) and \eqref{bounds-on-f-fprime} one obtains
a uniform bound on the characteristic speed $kf'(q)$.

\paragraph{Bound on the total variation of $q$.}
By definition of the total variation 
\begin{eqnarray*}
  \tv \{q(t,\cdot)\}
  &\dot=& \lim_{h\to 0+} \frac 1{h} \int_{-\infty}^{0} |q(t,x) - q(t,x-h)|\,dx \,,
\end{eqnarray*}
we have, for  any $h>0$
\begin{eqnarray}\label{L1-variation_vs_BV}
\frac 1{h} \int_{-\infty}^{0} |q(t,x) - q(t,x-h)|\,dx&\le&\tv \{q(t,\cdot)\}\,.
\end{eqnarray}

The total variation of $q$ does not change at time $t_n$ when $k$ is updated. 
Now consider a time interval $t \in [t_n,t_{n+1})$, and we estimate the change of the 
total variation of $q$ in this time interval.
We have 
\begin{eqnarray}
  \int_{-\infty}^{0} |q(t_{n+1},x) - q(t_{n+1},x-h)|\,dx
  &\le&\int_{-\infty}^{0} |q(t_n,x) - q(t_n,x-h)|\,dx \nonumber \\
&&   + \int_{t_n}^{t_{n+1}} \mathcal{E}(\tau)\,d\tau \label{L1-variation}
\end{eqnarray}
where 
\begin{eqnarray*}
  \mathcal{E}(\tau)  &=&  \limsup_{\theta\to 0+}
  \frac{\int_{-\infty}^0 \left| q(\tau+\theta,x-h) - \hat q(\tau+\theta,x) \right| \, dx 
  }{\theta}\,.
\end{eqnarray*}
Here  $\hat q$ is the entropy
solution to
\begin{eqnarray*} 
\left\{
\begin{array}{ll}
u_t + (k_n(x)f(u))_x=0\,,&\ \ t\ge \tau\,,\ x<0\\
u(\tau,x) = q(\tau, x-h)\,.
\end{array}
  \right.
\end{eqnarray*}
On the other hand, $q(t,x-h)$ is a solution of
\begin{eqnarray*} 
\left\{
\begin{array}{ll}
u_t + (k_n(x-h)f(u))_x=0\,,&\ \ t\ge \tau\,,\ x<0\\
u(\tau,x) = q(\tau, x-h)\,.
\end{array}
  \right.
\end{eqnarray*}
Using the estimate \eqref{estim:KlaRis}
we find
\begin{eqnarray*} 
\mathcal{E}(\tau)&\le& \|f\|_\infty \tv\{ k_n(\cdot-h) - k_n(\cdot) \}
~+~ C \left(1+ \tv \{q(\tau,\cdot)\}\right) \|k_n(\cdot-h) - k_n(\cdot)\|_\infty
\end{eqnarray*}
for a suitable constant $C$. Notice that
\begin{eqnarray*} 
  \left|k_n(x-h) - k_n(x)\right|&=& \left|\int_{x-h}^{x}(k_n) _x(\tau,y)\,dy \right|
  ~\le~  h \|k_nf\|_\infty
\end{eqnarray*}
and that
\begin{eqnarray}
  && \hspace{-1cm}
  \tv\{ k_n(\cdot-h) -k_n(\cdot)\} 
~=~  \left\| (k_n)_x(\cdot-h)- (k_n)_x(\cdot)\right\|_{\L^1} 
  \nonumber\\[1mm]
  &\le & 
  \left\|\left(k_n(\cdot-h)- k_n(\cdot)\right)\,f\left(q(\tau,\cdot)\right)
  \right\|_{\L^1}
  + \left\| k_n(\cdot-h) \cdot
    \left(f\left(q(\tau,\cdot-h)\right)- f\left(q(\tau,\cdot)\right)
    \right) \right\|_{\L^1}
  \nonumber\\[2mm]
  &\le& 
  h \|k_n f\|_\infty \cdot \left\|
      f\left(q(\tau,\cdot)\right)\right\|_{\L^1}
    ~+~ \left\| k_n\right\|_{\L^\infty} \|f'\|_\infty\,
    \left\|q(\tau,\cdot)- q(\tau,\cdot-h)\right\|_{\L^1}\,.\nonumber 
\end{eqnarray}
In conclusion, using also \eqref{L1-variation_vs_BV}, we obtain
\begin{eqnarray*}
  \mathcal{E}(\tau)  &\le& h \left\{M_1 + M_2 \tv \{q(\tau,\cdot) \}
    + M_3 \frac 1{h} \left\|q(\tau,\cdot)- q(\tau,\cdot-h)\right\|_{\L^1} \right\}\\
  &\le&  h \left\{M_1 + \left( M_2 + M_3\right) \tv \{q(\tau,\cdot) \}  \right\}
\end{eqnarray*}
where $M_i$ depend only on a-priori bounded quantities. Now from
\eqref{L1-variation} we obtain
\begin{eqnarray}\label{bound_TV}
  \tv \{q(t_{n+1},\cdot)\}   &\le&  \tv \{q(t_n,\cdot)\} ~+~ \int_{t_n}^{t_{n+1}} 
  \Big[ M_1 + \left(M_2 + M_3\right) \tv \{q(\tau,\cdot)\} \Big]\,d\tau\,.
\end{eqnarray}
We conclude that the total variation of $q$ may grow
exponentially in $t$ on each interval $(t_n,t_{n+1})$,  
but it remains bounded for any bounded time $t$.

\paragraph{Convergence to weak solutions; Existence of BV solutions.}
Now, without causing confusion, we will use $q^\Delta(t,x)$ for the
approximate solution, where $\Delta=\Delta t$ is the step size.  Let
$k^\Delta$ be the approximated coefficient of the equation, defined in
\eqref{k-Deltat}.

By compactness, there exists a subsequence of $\{q^\Delta(t,x)\}$, as
$\Delta \rightarrow 0$, that converges to a limit function $q(t,x)$ in
$\L^1_{loc}$. Let $k(t,x)$ be the integral term, \eqref{k},
corresponding to $q$, which is uniformly bounded as well as the
$k^\Delta$. We have
\begin{eqnarray*}
  k^\Delta(t,x)-k(t,x)&=&\O(1) 
  \left\{\int_x^0f(q^\Delta(t_n,\xi))\,d\xi - \int_x^0f(q(t,\xi))\,d\xi \right\}\\
  &=&\O(1) 
  \left\{
    \sup_\tau \tv\{ f(q^\Delta(\tau,\cdot)) \}\sup \dot x \, \Delta
    + 
    \int_x^0[f(q^\Delta(t,\xi))-f(q(t,\xi))]\,d\xi 
  \right\}
\end{eqnarray*}
that vanishes as $\Delta \to 0$.  Therefore we can pass to the limit
in the weak formulation. On the interval $[t_n,t_{n+1}]$,
$q^\Delta(t,x)$ satisfies
$$
\int_{t_n}^{t_{n+1}} \int_{-\infty}^0 
(q^\Delta \phi_t + k^\Delta f(q^\Delta) \phi_x)\, dx \, dt 
= \int_{-\infty}^0 \left[ q^\Delta \phi(t_{n+1},x) - q^\Delta \phi(t_n,x) \right] \, dx 
$$
for some test function $\phi$ with compact support inside $[0,T]\times \R_- $. 
Summing this up over $n$, we get
\begin{eqnarray}\label{weak-form-Delta}
\int_0^T 
\int_{-\infty}^0 
(q^\Delta \phi_t + k^\Delta f(q^\Delta) \phi_x)\, dx \, dt 
&=&
\int_{-\infty}^0 \left[ q^\Delta \phi(T,x) - q^\Delta \phi(0,x) \right] \, dx\,.
\end{eqnarray}
Since $q^\Delta\to q$ in $\L^1_{loc}$, $f(q^\Delta)\to f(q)$ in
$\L^1_{loc}$, $k^\Delta\to k$ pointwise and $k^\Delta$, $k$ are
uniformly bounded, by dominated convergence we can take the limit as
$\Delta\to 0$ and have the convergence of \eqref{weak-form-Delta} to
\begin{eqnarray*}
\int_0^T \int_{-\infty}^0 
\left[q(t,x) \phi_t(t,x) + k(t,x) f(q(t,x)) \phi_x(t,x)\right]\, dx \, dt 
&=&
\int_{-\infty}^0 \left[ q \phi(T,x) - q \phi(0,x) \right] \, dx\,.
\end{eqnarray*}
This completes the proof of existence of BV solutions for (\ref{eq1}).
\end{proof}

\subsection{Global existence of BV solutions}
\label{subsec3.1}

Once the BV solutions exist locally in time, we can further show that
they enjoy better properties than the ones deduced from the
approximate solutions.  In particular we show that the lower and upper
bounds on $q$ do not depend on time $t$, leading to global in time
existence of BV solutions.

Let $q$ be an entropy weak solution of \eqref{the-problem} on $[0,T]\times\R_-$.
We will now improve the needed bounds. 

\paragraph{Lower  bound on $q$.}
Given any point $(\bar t,
\bar x)\in (0,T)\times \R_-$, let $t\to x(t)$ be the minimal backward
characteristic (which is classical), defined for $t\in[0,\bar t]$. By
setting $q(t) = q(t,x(t))$, we have
\begin{eqnarray}
\left\{
\begin{array}{rcl}
  x'(t)&=&k(t,x) f'(q(t))\,,\\[2mm]
  q'(t) &=&   -k_x(t,x(t)) f(q) = k f(q)^2 \ge 0\,,
\end{array}
\right.
&&
\begin{array}{l}
x(\bar t)=\bar x\,,\\[2mm]
q(\bar t) = q(\bar t, \bar x-)\,.
\end{array}
\label{eq:q-along-chars}  
\end{eqnarray}
We see that the solution $q$ is non-decreasing along any
characteristics.  Therefore, we have $\inf_x q(t,x)\ge \inf \bar q(x)
\ge \kappa_0 > -1$ for all $t\ge0$.

\paragraph{Upper bound on $q$.}
Again, consider a point $(\bar t, \bar x)$ and let $t\to x(t)$ be 
the minimal backward characteristic through it.
{From} the second equation in (\ref{eq:q-along-chars}) we see that if $q(0,t(0)) \le 0$, then 
$q\to 0$ as $t\to+\infty$. 
Now consider $q(0,x(0)) >0$, and we have $q(t,x(t))\ge 0$ for all $t$. 
Define 
\begin{equation}\label{W}
 W(t,x) = \int_{-\infty}^x|q(t,y)|\,dy 
\,, \qquad x<0\,,
\end{equation}
that satisfies
$$
0~\le~ W(t,x) ~\le~ \|q(0,\cdot)\|_{L^1(\R_-)} \,.
$$
Using (\ref{kruzkov}) with $\alpha=1$, we have 
$$
W_t ~=~ \int_{-\infty}^x|q(t,y)|_t \,dy 
~\le ~ -\int_{-\infty}^x \Big( k(t,x)\, |f(q)| \Big)_x \,dy 
~=~ - k |f(q)|\,.
$$
The variation of $W$ along the characteristic is 
\begin{eqnarray}
  \frac{d}{dt} W(t,x(t))  &=&    W_t + x' W_x
  ~\le~ -k|f| +  |q| k f' ~=~ 
  k\left( -|f| + |q|f'(q) \right)
\nonumber \\
  &= & k\left( -f + q f'(q) \right)
  = - f^2 k \frac{f - qf'(q)}{f^2} 
  = - \frac{d}{dt} \left( \frac{q(t,x(t))}{f(q(t,x(t)))}\right) 
  \,.
  \label{W-along-chars}
\end{eqnarray}
Here we remove the absolute value signs because $q>0$. 
Then, (\ref{W-along-chars}) implies that 
$$
W(t,x(t)) + \frac{q(t,x(t))}{f(q(t,x(t)))} \equiv C
$$
along characteristics. This gives the bound
\begin{equation}\label{qBound}
\frac{q(t,x(t))}{f(q(t,x(t)))} 
= \frac{q(0,x(0))}{f(q(0,x(0)))} + W(0,x(0)) - W(t,x(t)) \le C_1\,,
\end{equation}
where $C_1$ can be chosen independently of $(\bar t,\bar
x)$. Recalling \eqref{eq:f}, we have
$$ 
\lim_{q\to +\infty} \frac{q}{f(q)} = +\infty\,.
$$
Therefore, (\ref{qBound}) implies an upper bound for $q$ for all $t$.
The uniform bound on the total variation follows because the constants
$M_i$ in \eqref{bound_TV} are now bounded uniformly in time.

\subsection{Continuous dependence from the data for the
  integro-differential equation}
\label{Subsec:3}

In this section we prove the last part of Theorem~\ref{th:1}, showing
that the flow generated by the integro-differential equation
(\ref{eq1}) is Lipschitz continuous, restricted to any domain
$\D\subset\L^1(\R_-)$ of functions $q(\cdot)$ satisfying the following
uniform bounds in \eqref{def:calD}, for some constants $C_0$,
$\kappa_0$.

Consider two solutions $q_1(t,\cdot)$, $q_2(t,\cdot)$ of the
integro-differential equation (\ref{eq1}), say with
initial data
$$
q_1(0,x)= \bar q_1(x)\,,\qquad\qquad q_2(0,x) =\bar q_2(x)\qquad\qquad
x<0\,,
$$
and satisfying the conditions in (\ref{def:calD}) for $t\in[0,T]$. We
are going to prove that
\bel{7.9} \|q_1(t,\cdot)- q_2(t,\cdot)\|_{\L^1(\R_-)}
\leq \|\bar q_1-\bar q_2\|_{\L^1(\R_-)} + L \cdot \int_0^t
\|q_1(s,\cdot)- q_2(s,\cdot)\|_{\L^1(\R_-)}\,ds\,, 
\eeq
for a suitable constant $L$. By Gronwall lemma, this yields
\eqref{continous-dep-on-init-data}, hence the Lipschitz continuous
dependence of solutions of (\ref{eq1}) on the initial data.

Define the functions $k_1(t,x)$, $k_2(t,x)$ as in \eqref{k},
corresponding to $q_1(t,x)$, $q_2(t,x)$ respectively.
Now set
$$
k^\theta(t,x)\doteq
\left\{ \begin{array}{l}
k_1(t,x)\quad\hbox{if}~~t\in [0,\theta]\,, \\
[2mm]
k_2(t,x)\quad\hbox{if}~~t>\theta\,.
\end{array}\right.
$$
Finally, for any given $\theta\in [0,T]$, let $q^\theta=q^\theta
(t,x)$ be the solution of the conservation law
\begin{equation}\label{7.5} q_t + \left( k^\theta(t,x)\,
f(q)\right)_x=0\,,\qquad
q^\theta(0,x)= \bar q_2(x)\,.
\end{equation}
Observe that, for each fixed $\theta$, the distance between any two
entropy-admissible solutions of the conservation law (\ref{7.5}) is
non-increasing in time. In particular, for $\theta=T$, call $\hat q$
the solution of
$$
q_t + \left(k_1(t,x)\,f(q)\right)_x=0\,,
$$
with initial data $\hat q(0,x)= \bar q_2(x)$ (see Figure \ref{figL}).
We have
\bel{7.6}
\left\|q_1(t,\cdot)-\hat q(t,\cdot)\right\|_{\L^1(\R_-)}
~\leq~\left\|\bar q_1-\bar q_2\right\|_{\L^1(\R_-)}
\qquad\qquad\forall t\in \left[0,T\right]\,.
\eeq

\begin{figure}[htb]
\centerline{\includegraphics[width=10cm]{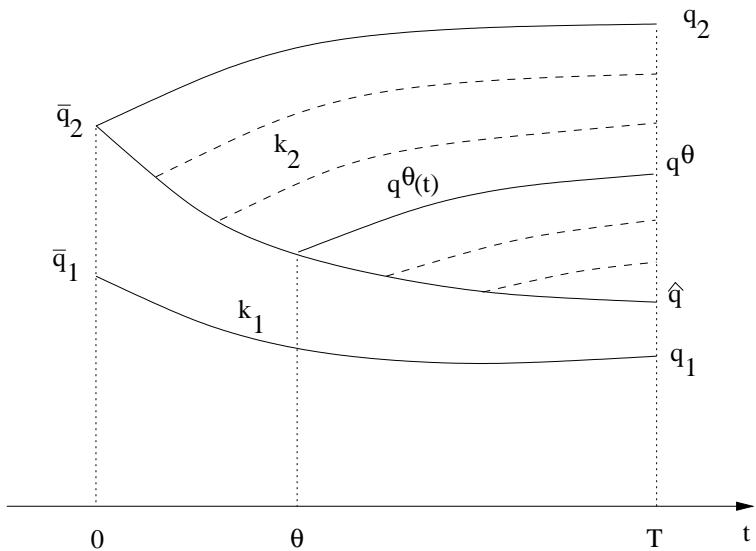}}
  \caption{The flow of solutions $q_1,\hat q, q^\theta, q_2$
           for the integro-differential equation.}
  \label{figL}
\end{figure}

Moreover we can use the Lipschitz property of the solution operator
for (\ref{eq1}) with $k=k_2$ fixed, and get the
distance estimate
\bel{7.7} \left\|\hat q(T,\cdot)- q_2(T,\cdot)\right\|_{\L^1(\R_-)}
~\leq~ \int_0^T E(\tau)\,d\tau\,,
\eeq
where
$$
E(\tau)  \doteq  \limsup_{h\to 0+}
\frac{\left\| q^{\tau}(\tau+h,\cdot)-\hat q(\tau+h,\cdot)\right\|_{\L^1}}{h}\,.
$$
Indeed, observe that $\hat q(\tau,\cdot) =q^\theta(\tau,\cdot)$
whenever $\tau\leq\theta$, for any $\tau\in [0,T]$.

To compute the integrand in (\ref{7.7}), observe that the functions
$h\mapsto q^{\tau}(\tau+h,\cdot)$ and $h\mapsto \hat q(\tau+h,\cdot)$
take the same value $\hat q(\tau,\cdot)$ when $h=0$, and $h\mapsto
q^{\tau}(\tau+h,x)$ satisfies the conservation law
\bel{7.8}
q_h + \left(  k_2(\tau+h,x)\,f(q) \right)_x=0\,,
\eeq
while $h\mapsto \hat q(\tau+h,x)$ solves
\bel{7.8b} q_h + \left( k_1(\tau+h,x)\,f(q)\right)_x=0\,,
\eeq
for $h\ge 0$.  By using \eqref{dep_on_coefficients} in
Theorem~\ref{3}, we can measure the error term $E(\tau)$.  By the
facts that $\left\|q^\tau (\tau,\cdot)\right\|_{\L^\infty}$,
$\left\|\hat q(\tau,\cdot)\right\|_{\L^\infty}$, $\tv
\{q^\tau(\tau,\cdot)\}$, $\tv \{\hat q(\tau,\cdot)\}$, $\tv
\{k_1(\tau,\cdot)\}$ and $\tv \{k_2(\tau,\cdot)\}$ are all bounded,
the coefficients $\Hat C_1$ and $\Hat C_2$ in
\eqref{dep_on_coefficients} are all bounded constants.  Let $M$ be a
generic bounded constant, we get
$$ 
\left\|q^\tau(\tau+h,\cdot)-\hat q(\tau+h,\cdot)\right\|_{\L^1} 
\le  M h \left[ 
\sup_{\tau\le t\le \tau+h} \tv( k_1(t,\cdot) -k_2(t,\cdot)) 
+  \left\|k_1 -k_2 \right\|_{\L^\infty([\tau,\tau+h]\times\R_-)} 
\right]\,.
$$
Therefore, we have
\begin{equation}\label{E0}
E(\tau)=
M\cdot \tv\{k_1(\tau,\cdot)- k_2(\tau,\cdot)\} 
+ M\cdot \left\|k_1(\tau,\cdot)
  - k_2(\tau,\cdot)\right\|_{\L^\infty}
\end{equation}
Recalling the definitions of $k_1$, $k_2$ we deduce that
\begin{eqnarray}
\left\|k_1(\tau,\cdot)- k_2(\tau,\cdot)\right\|_{\L^\infty} &=&
M \cdot \sup_{x<0} \left|\int_x^0
f\left(q_1(\tau,y)\right)
\,dy -
\int_x^0
f\left(q_2(\tau,y)\right)
\,dy \right|
\nonumber\\[2mm]
&=& M\cdot \|q_1(\tau,\cdot) - q_2(\tau,\cdot)\|_{\L^1}\,,\label{E1}
\end{eqnarray}
and, using also (\ref{E1}),
\begin{eqnarray}
&& \hspace{-1cm}
\tv\{ k_1(\tau,\cdot) -k_2(\tau,\cdot)\} 
~=~
\left\| (k_1)_x(\tau,\cdot)- (k_2)_x(\tau,\cdot)\right\|_{\L^1} 
\nonumber\\[1mm]
&=& 
\left\| k_1(\tau,\cdot) \,f\left(q_1(\tau,\cdot)\right)
- k_2(\tau,\cdot)\,f\left(q_2(\tau,\cdot)\right)\right\|_{\L^1}
 \nonumber\\[2mm]
&\le & 
\left\| (k_1(\tau,\cdot)- k_2(\tau,\cdot))\,f\left(q_1(\tau,\cdot)\right)
\right\|_{\L^1}
+ \left\|  k_2(\tau,\cdot) \cdot
\left(f\left(q_1(\tau,\cdot)\right)- f\left(q_2(\tau,\cdot)\right)
\right) \right\|_{\L^1}
\nonumber\\[2mm]
&=& 
\left\| k_1(\tau,\cdot)- k_2(\tau,\cdot)\right\|_{\L^\infty}\cdot \left\|
f\left(q_1(\tau,\cdot)\right)\right\|_{\L^1}
~+~ \left\| k_2(\tau,\cdot)\right\|_{\L^\infty}
\left\| q_1(\tau,\cdot)- q_2(\tau,\cdot)\right\|_{\L^1}\nonumber \\[2mm]
&=&
M \left\|q_1(\tau,\cdot)- q_2(\tau,\cdot) \right\|_{\L^1}\,.
\label{E2}
\end{eqnarray}

Putting the estimates \eqref{E1} and \eqref{E2} into \eqref{E0}, we get
$$
E(\tau)~\leq~ L\cdot \|q_1(\tau,\cdot)- q_2(\tau,\cdot)\|_{\L^1}
$$
for a suitable constant $L$.  Inserting this estimate in
(\ref{7.7}) and using (\ref{7.6}) one finally obtains (\ref{7.9}).

\appendix
\section{Properties of the integral operator}\label{app:propr_K}
\setcounter{equation}{0}

In this Appendix we prove some properties of the integral term $k$ in
terms of a Lipschitz flow $t\mapsto q(t,\cdot)$. The operator $K$, see
\eqref{K}, is defined on the set
\begin{eqnarray*}
\left\{q \in \L^1(\R_-)\cap BV(\R_-)\,;\quad \inf_{x<0}\, q(x) >-1 \right\}
\end{eqnarray*}
and valued in $Lip(\R_-)$. Its properties are summarized in the
following Proposition.
\begin{proposition}\label{properties_of_k}
  Let $C_0$, $\kappa_0$, $T$ be given positive constants. Assume that
  the map $q:[0,T]\to \D_{C_0,\kappa_0}$ is Lipschitz continuous as a
  function in $\L^1(\R_-)$.

  Define $k$ as in \eqref{k}. Then
\begin{eqnarray*}
  \mbox{{\bf (K)}}&&\left\{
\begin{array}{l}
  k(t,x): [0,T]\times \R_- \to \R_+ \mbox{ is
    bounded and Lipschitz continuous, with}\,\inf_{t,x} k >0\,;
  \\[2mm]
  \tv k(t,\cdot)\,,\ \tv k_x(t,\cdot) \mbox{ are bounded uniformly in time;}\\[2mm]
  [0,T]\ni t\to k_x(t,\cdot)\in{\L^1(\R_-)}\mbox{  is Lipschitz
    continuous.}
\end{array}
\right.
\end{eqnarray*}
\end{proposition}

\begin{proof}
  To begin, notice that the quantity $k$ is well-defined and is
  Lipschitz continuous on $[0,T]\times\R_-$.

  Let $L$ be a Lipschitz constant of the map $[0,T]\ni t\mapsto
  q(t)\in\L^1(\R_-)$. 
  {From} the bounds \eqref{def:calD} one easily
  deduces that
\begin{eqnarray}
 \|q(t,\cdot)\|_{\L^\infty(\R_-)}&\le&C_0\,, \label{apriori_q_infty} \\
  \|f(q(t,\cdot))\|_{\L^\infty(\R_-)}&\le& 
  \max\left\{  \left|f(C_0)\right|, \left|f(\kappa_0)\right| \right\}\,,
  \label{apriori_f_infty} \\
\|f(q(t,\cdot))\|_{\L^1(\R_-)}&\le& 
|f'(\kappa_0)| \cdot \|q(t,\cdot)\|_{\L^1(\R_-)}  \le C_0|f'(\kappa_0)| \,,
\label{L1_norm_of_f}\\
\|f(q(t_1,\cdot)) - f(q(t_2,\cdot)) \|_{\L^1(\R_-)} &\le& 
L |f'(\kappa_0)| \cdot  |t_1 - t_2|\,. 
\label{int_f_lip}
\end{eqnarray}

By the assumptions on $q$ we find that
$$
\left| \int_{x}^0f(q(t,\xi))\,d\xi \right| ~\leq~
\|f(q(t,\cdot))\|_{\L^1(\R_-)} ~\leq~  C_0|f'(\kappa_0)|\,.
$$
Hence the integral term $k$ is bounded and satisfies
$$
0<\exp\left({-C_0|f'(\kappa_0)|}\right) \le k(t,x) \le \exp\left({C_0|f'(\kappa_0)|}\right)\,.
$$
Moreover, for all $0\le t_1 < t_2$ we have
\begin{eqnarray}\nonumber
\left|\int_x^0  \left[f(q(t_1,\xi)) - f(q(t_2,\xi))\right]
\,d\xi\right| &\le&  \|f(q(t_1,\cdot)) - f(q(t_2,\cdot))\|_{\L^1(\R_-)}
  ~\le~ L |f'(\kappa_0)| \cdot  |t_1 - t_2|\,.
  \nonumber
\end{eqnarray}
This leads to  the Lipschitz continuity in $t$ for $k(t,x)$. Namely, for all $x$ we have
\begin{equation}\label{k_lip_in_time}
  \left| k(t_1,x) - k(t_2,x) \right| = \O(1)
  \left|\int_x^0  \left[f(q(t_1,\xi)) - f(q(t_2,\xi))\right]
    \,d\xi\right|
  \le\Hat L  ~~ |t_1-t_2|\,.
\end{equation}
Here the Lipschitz constant $\Hat L$ depends on the parameters $L$, $C_0$, $\kappa_0$.

{From} the definition of $k$,  the derivative function $k_x$ satisfies
\begin{equation}\label{eq:kx}
  k_x = - k  f(q) \,\,\in\, \L^1\cap \L^\infty\,.
\end{equation}
This immediately shows three facts: (i) $k(t,x)$ is Lipschitz in space
variable $x$, (ii) $k(t,\cdot)\in BV(\R_-)$ where the BV bounds are
uniform in $t$, and (iii) $k_x(t,\cdot)\in BV(\R_-)$.

{From} \eqref{eq:kx} we get the estimate on the total variation of
$k_x$
$$
  \tv (k_x) \le
  \tv (k) \cdot \| f(q)\|_{\L^\infty(\R_-)} ~+~ \| k\|_{\L^\infty(\R_-)}\tv (f(q(t,\cdot)))
  \le M \, \tv (q)\,,
$$
with $M$ depending on the parameters $L,C_0,\kappa_0$.

\medskip Finally, we show that $[0,T]\ni t\to
k_x(t,\cdot)\in\L^1(\R_-)$ is Lipschitz continuous. By using
\eqref{k_lip_in_time}, \eqref{L1_norm_of_f} and \eqref{int_f_lip}, one
has
\begin{eqnarray*}
 \hspace{-2cm} \|k_x(t_1,\cdot) - k_x(t_2,\cdot)\|_{\L^1(\R_-)} 
&= &\tv \left\{k(t_1,\cdot) - k(t_2,\cdot)\right\}\\
&=&  \|k(t_1,\cdot) f(q(t_1,\cdot)) - k(t_2,\cdot)f(q(t_2,\cdot))\|_{\L^1(\R_-)}\\
&\le&   \|k(t_1,\cdot) - k(t_2,\cdot)\|_{\L^\infty(\R_-)}  \|f(q(t_1,\cdot)) \|_{\L^1(\R_-)}\\
&&  + ~\|k(t_2,\cdot)\|_{\L^\infty(\R_-)} \|f(q(t_1,\cdot)) - f(q(t_2,\cdot))\|_{\L^1(\R_-)}\\
& \le& \Hat M |t_1 - t_2|
\end{eqnarray*}
with $\Hat M$ depending on the parameters $L$, $C_0$, $\kappa_0$.
\end{proof}

\paragraph{Acknowledgement.}
This paper was started as part of the international research program
on Nonlinear Partial Differential Equations at the Centre for Advanced
Study at the Norwegian Academy of Science and Letters in Oslo during
the academic year 2008--09. The first author would like to acknowledge
also the kind hospitality of the Department of Mathematics, University
of Ferrara. The work of the second author is partially supported by
NSF grant DMS-0908047.

\renewcommand{\baselinestretch}{1}


\begin{thebibliography}{99}


\bibitem{AGG2} Amadori, D., Gosse, L. and Guerra, G.; Godunov-type
  approximation for a general resonant balance law with large data.
  \textit{J. Differential Equations} \textbf{198} (2004), 233--274

\bibitem{AS2} Amadori, D. and Shen, W.; The Slow Erosion Limit in a
  Model of Granular Flow, \textit{Arch. Ration. Mech. Anal.},
  \textbf{199} (2011), 1--31

\bibitem{AS4} Amadori, D. and Shen, W.; Front Tracing Approximations
  for Slow Erosion in Granular Flow. Preprint 2010

\bibitem{BJ} Baiti, P. and Jenssen, H. K.; Well-posedness for a class
  of $2\times2$ conservation laws with $L\sp\infty$ data.
  \textit{J. Differential Equations} \textbf{140} (1997), 
  161--185

\bibitem{BPZ} Bressan, A. and Zhang, P. and Zheng, Y.; 
\textit{Asymptotic variational wave equations}. 
Arch. Ration. Mech. Anal. \textbf{183} (2007), 163--185

\bibitem{CH} Camassa, R. and Holm, D.;
\textit{An integrable shallow water equation with peaked solitons}. 
Phys. Rev. Lett. \textbf{71} (1993), 1661--1664

\bibitem{Ch-Chr} Chen, G.-Q. and Christoforou, C.: Solutions for a
  nonlocal conservation law with fading memory.
  \textit{Proc. Amer. Math. Soc.} {\bf 135} (2007), 3905--3915

\bibitem{CoHeMe} Colombo, R.M., Herty, M. and Mercier, M.; Control of
  the Continuity Equation with a Non Local Flow.  To appear on
  \textit{ESAIM COCV}

\bibitem{CoMeRo} Colombo, R.M., Mercier, M. and Rosini M.; Stability
  and total variation estimates on general scalar balance laws.
  \textit{Commun. Math. Sci.} \textbf{7} (2009), 
  37--65

\bibitem{Daf88} Dafermos, C.M.; Solutions in $L^\infty$ for a
  conservation law with memory. \textit{Analyse mathématique et
    applications}, 117--128, Gauthier-Villars, Montrouge, 1988

\bibitem{Gue04} Guerra, G.; Well-posedness for a scalar conservation
  law with singular nonconservative source.  \textit{J. Differential
    Equations} \textbf{206} (2004), 
  438--469

\bibitem{HK} Hadeler, K.P. and Kuttler, C.; Dynamical models for
  granular matter. \textit{Granular Matter} \textbf{2} (1999), 9--18

\bibitem{KaRi} Karlsen, K.-H. and Risebro, N.-H.; On the uniqueness
  and stability of entropy solutions of nonlinear degenerate parabolic
  equations with rough coefficients. \textit{Discrete
    Contin. Dyn. Syst.} \textbf{9} (2003), 1081--1104

\bibitem{KlRi}{Klausen, R. A. and Risebro, N. H.}; Stability of
  conservation laws with discontinuous coefficients.
  \textit{J. Differential Equations} \textbf{157} (1999), 
  41--60
 
\bibitem{Kru} Kru{\v{z}}kov, S.N.; {First order quasilinear equations
    with several independent variables}. \textit{Mat. Sb. (N.S.)}
  \textbf{81} (123) (1970), 228--255

\bibitem{LTW} Lin, L. and Temple, J.~B. and Wang, J.; Suppression of
  oscillations in Godunov's method of a resonant non-strictly
  hyperbolic system, \textit{SIAM J. Numer. Anal.} \textbf{32}(3)
  (1995), 841--864

\bibitem{SZ} Shen, W. and Zhang, T.; Erosion Profile by a Global Model
  for Granular Flow. Preprint 2010, accepted for publication on
  \textit{Arch. Ration. Mech. Anal.}

\end{thebibliography}
\end{document}